\documentclass[12pt, reqno]{amsart}
\usepackage{amsmath, amsthm, amscd, amsfonts, amssymb, graphicx, color}
\usepackage[bookmarksnumbered, colorlinks, plainpages]{hyperref}
\hypersetup{colorlinks=true,linkcolor=red, anchorcolor=green, citecolor=cyan, urlcolor=red, filecolor=magenta, pdftoolbar=true}

\newtheorem{theorem}{Theorem}[section]
\newtheorem{lemma}[theorem]{Lemma}

\newtheorem{corollary}[theorem]{Corollary}
\theoremstyle{definition}

\theoremstyle{remark}

\numberwithin{equation}{section}

\begin{document}

\setcounter{page}{1}

\title[]{Characterizing linear mappings through zero products or zero Jordan products}

\author[G. An, J. He and L. Ji]{ Guangyu An$^1$, Jun He$^2$  and  Jiankui  Li$^{3*}$}

\address{$^{1}$ Department of Mathematics, Shaanxi University of Science and Technology,
Xi'an 710021, China.}
\email{\textcolor[rgb]{0.00,0.00,0.84}{anguangyu310@163.com}}
\address{$^{2}$ Department of Mathematics, Anhui Polytechnic University,
Wuhu 241000, China.}
\email{\textcolor[rgb]{0.00,0.00,0.84}{hejun$_{-}$12@163.com}}
\address{$^{3*}$ Department of Mathematics, East China University of
Science and Technology, Shanghai 200237, China.}
\email{\textcolor[rgb]{0.00,0.00,0.84}{jiankuili@yahoo.com}}


\keywords{$*$-(Jordan) derivation, $*$-(Jordan) left derivation,
zero (Jordan) product determined algebra, $C^*$-algebra, von Neumann algrbra.
\newline \indent $^{*}$ Corresponding author}

\begin{abstract}
Let $\mathcal{A}$ be a $*$-algebra and $\mathcal{M}$ be a $*$-$\mathcal A$-bimodule,
we study the local properties of
$*$-derivations and $*$-Jordan derivations
from $\mathcal{A}$ into $\mathcal{M}$ under the following orthogonality conditions on elements
in $\mathcal A$: $ab^*=0$, $ab^*+b^*a=0$ and $ab^*=b^*a=0$.
We characterize the mappings on zero product determined algebras and zero Jordan product determined algebras.
Moreover, we give some applications on $C^*$-algebras, group algebra, matrix algebras, algebras of locally measurable operators
and von Neumann algebras.
\end{abstract}\maketitle


\section{Introduction}

Throughout this paper, let $\mathcal{A}$ be an associative algebra
over the complex field $\mathbb{C}$ and $\mathcal{M}$ be an $\mathcal{A}$-bimodule.
For each $a$, $b$ in $\mathcal A$, we define the \emph{Jordan product} by
$a\circ b=ab+ba$.
A linear mapping $\delta$ from $\mathcal{A}$ into $\mathcal{M}$ is
called a \emph{derivation} if $\delta(ab)=a\delta(b)+\delta(a)b$ for each $a,b$ in $\mathcal{A}$;
and $\delta$ is called a \emph{Jordan derivation} if $\delta(a\circ b)=a\circ\delta(b)+\delta(a)\circ b$
for each $a,b$ in $\mathcal{A}$.
It follows from the results in \cite{Cuntz, hejazian, Johnson1} that
every Jordan derivation from a $C^*$-algebra into its Banach bimodule is a derivation.

By an \emph{involution} on an algebra $\mathcal A$, we mean a mapping
$*$ from $\mathcal{A}$ into itself, such that
$$(\lambda a+\mu b)^*=\bar{\lambda}a^*+\bar{\mu}b^*,~(ab)^*=b^*a^*~\mathrm{and}~(a^*)^*=a,$$
whenever $a,b$ in $\mathcal{A}$, $\lambda,\mu$ in $\mathbb{C}$ and $\bar{\lambda},\bar{\mu}$
denote the conjugate complex numbers.
An algebra $\mathcal A$ equipped with an involution is called a $*$-algebra. Moreover,
let $\mathcal A$ be a $*$-algebra, an $\mathcal A$-bimodule $\mathcal M$ is called a $*$-$\mathcal A$-bimodule
if $\mathcal M$ equipped with a $*$-mapping from $\mathcal M$ into itself, such that
$$(\lambda m+\mu n)^*=\bar{\lambda}m^*+\bar{\mu}n^*,~(am)^*=m^*a^*,~(ma)^*=a^*m^*~\mathrm{and}~(m^*)^*=m,$$
whenever $a$ in $\mathcal{A}$, $m,n$ in $\mathcal{M}$ and $\lambda,\mu$ in $\mathbb{C}$.
An element $a$ in $\mathcal{A}$ is called \emph{self-adjoint} if $a^*=a$;
an element $p$ in $\mathcal{A}$ is called an \emph{idempotent} if $p^2=p$; and $p$ is called a \emph{projection}
if $p$ is both a self-adjoint element and an idempotent.

In \cite{Kishimoto}, A. Kishimoto studies the $*$-derivations on a $C^*$-algebra.
Let $\mathcal A$ be a $*$-algebra and $\mathcal M$ be a $*$-$\mathcal A$-bimodule.
A derivation $\delta$ from $\mathcal{A}$ into $\mathcal{M}$ is called a
\emph{$*$-derivation}  if $\delta(a^*)=\delta(a)^*$ for every $a$ in $\mathcal A$.
Obviously, every derivation $\delta$ is a linear combination of two
$*$-derivations. In fact,
we can define a linear mapping
$\hat{\delta}$ from $\mathcal{A}$ into $\mathcal{M}$ by $\hat{\delta}(a)=\delta(a^*)^*$ for every
$a$ in $\mathcal A$, therefore $\delta=\delta_1+i\delta_2$,
where $\delta_1=\frac{1}{2}(\delta+\hat{\delta})$ and $\delta_2=\frac{1}{2i}(\delta-\hat{\delta})$.
It is easy to show that $\delta_1$ and $\delta_2$ are both $*$-derivations. Similarly, we
can define the $*$-Jordan derivations.

For $*$-derivations and $*$-Jordan derivations, in \cite{J. Alaminos2, B. Fadaee, Hoger Ghahramani1, Hoger Ghahramani2}, the authors
characterize the following two conditions
on a linear mapping $\delta$ from a $*$-algebra $\mathcal{A}$ into its $*$-bimodule $\mathcal{M}$:
\begin{align*}
&(\mathbb{D}_{1})~a,b\in\mathcal{A},~ab^*=0\Rightarrow a\delta(b)^*+\delta(a)b^*=0;\\
&(\mathbb{D}_{2})~a,b\in\mathcal{A},~ab^*=b^*a=0\Rightarrow a\delta(b)^*+\delta(a)b^*=\delta(b)^*a+b^*\delta(a)=0;
\end{align*}
where $\mathcal A$ is a $C^*$-algebra, a zero product determined algebra or a group algebra $L^1(G)$.

Let $\mathcal{J}$ be an ideal of $\mathcal{A}$, we say that $\mathcal{J}$
is a \emph{right separating set} or \emph{left separating set} of $\mathcal{M}$
if for every $m$ in $\mathcal{M}$,
$\mathcal{J}m=\{0\}$ implies $m=0$ or $m\mathcal{J}=\{0\}$ implies $m=0$, respectively.
We denote by $\mathfrak{J}(\mathcal{A})$
the subalgebra of $\mathcal A$ generated algebraically by all
idempotents in $\mathcal{A}$.

In Section 2, we suppose that $\mathcal A$ is a $*$-algebra and $\mathcal{M}$ is a $*$-$\mathcal A$-bimodule that
satisfy one of the following conditions:\\
$(1)$ $\mathcal A$ is a zero product determined Banach $*$-algebra with a bounded approximate identity and
$\mathcal M$ is an essential Banach $*$-$\mathcal A$-bimodule;\\
$(2)$ $\mathcal{A}$ is a von Neumann algebra and $\mathcal M=\mathcal A$;\\
$(3)$ $\mathcal A$ is a unital $*$-algebra and $\mathcal{M}$ is a unital $*$-$\mathcal A$-bimodule
with a left or right separating set $\mathcal{J}\subseteq\mathfrak{J}(\mathcal{A})$;\\
and we investigate whether the linear mappings from $\mathcal A$ into $\mathcal M$ satisfying the
condition $\mathbb{D}_{1}$ characterize $*$-derivations.
In particular, we generalize some results in \cite{B. Fadaee, Hoger Ghahramani1, Hoger Ghahramani2}.

An $\mathcal A$-bimodule $\mathcal{M}$ is said to have the \emph{property $\mathbb{M}$},
if there is an ideal $\mathcal{J}\subseteq\mathfrak{J}(\mathcal{A})$ of $\mathcal{A}$ such that
$$\{m\in\mathcal{M}: xmx=0~\mathrm{for~every}~x\in\mathcal{J}\}=\{0\}.$$
It is clear that if $\mathcal{A}=\mathfrak{J}(\mathcal{A})$, then $\mathcal{M}$ has
property $\mathbb{M}$.

For $*$-Jordan derivations, we can study the following conditions on
a linear mapping $\delta$ from a $*$-algebra $\mathcal A$ into its $*$-$\mathcal A$-bimodule $\mathcal M$:
\begin{align*}
&(\mathbb{D}_{3})~a,b\in\mathcal{A},~a\circ b^*=0\Rightarrow a\circ\delta(b)^*+\delta(a)\circ b^*=0.\\
&(\mathbb{D}_{4})~a,b\in\mathcal{A},~ab^*=b^*a=0\Rightarrow a\circ\delta(b)^*+\delta(a)\circ b^*=0.
\end{align*}
It is obvious that the condition $\mathbb{D}_{2}$ or $\mathbb{D}_{3}$ implies the condition $\mathbb{D}_{4}$.

In Section 3, we suppose that $\mathcal A$ is a $*$-algebra and $\mathcal{M}$ is a $*$-$\mathcal A$-bimodule that
satisfy one of the following conditions:\\
$(1)$ $\mathcal A$ is a unital zero Jordan product determined $*$-algebra and
$\mathcal M$ is a unital $*$-$\mathcal A$-bimodule;\\
$(2)$ $\mathcal A$ is a unital $*$-algebra and $\mathcal{M}$ is a unital $*$-$\mathcal A$-bimodule
such that the property $\mathbb{M}$;\\
$(3)$ $\mathcal A$ is a $C^*$-algebra (not necessary unital) and $\mathcal M$ is an essential Banach $*$-$\mathcal A$-bimodule;\\
and we investigate whether the linear mappings from $\mathcal A$ into $\mathcal M$ satisfying the
condition $\mathbb{D}_{3}$ or $\mathbb{D}_{4}$ characterize $*$-Jordan derivations.
In particular, we improve some results in \cite{B. Fadaee, Hoger Ghahramani1, Hoger Ghahramani2}.

\section{$*$-derivations on some algebras}

A (Banach) algebra $\mathcal{A}$ is said to be \emph{zero product determined}
if every (continuous) bilinear mapping $\phi$ from $\mathcal{A}\times\mathcal{A}$
into any (Banach) linear space $\mathcal{X}$ satisfying
$$\phi(a, b)=0,~\mathrm{ whenever } \  ab = 0$$
can be written as $\phi(a, b)= T(ab)$, for some (continuous) linear mapping $T$
from $\mathcal{A}$ into $\mathcal X$.
In \cite{M. Bresar 3}, M. Bre\v{s}ar shows that if $\mathcal A=\mathfrak{J}(\mathcal{A})$,
then $\mathcal A$ is a zero product determined, and in \cite{J. Alaminos},
the authors prove that every $C^*$-algebra $\mathcal A$ is zero product determined.

Let $\mathcal A$ be a Banach $*$-algebra and $\mathcal M$ be a
Banach $*$-$\mathcal A$-bimodule. Denote by $\mathcal M^{\sharp\sharp}$ the second dual
space of $\mathcal M$. In the following, we show that $\mathcal M^{\sharp\sharp}$
is also a Banach $*$-$\mathcal A$-bimodule.

Since $\mathcal M$ is a Banach $*$-$\mathcal A$-bimodule,
$\mathcal{M}^{\sharp\sharp}$ turns into a dual Banach $\mathcal{A}$-bimodule with the operation defined by
$$a\cdot m^{\sharp\sharp}=\lim\limits_{\mu}am_{\mu}~\mathrm{and}~m^{\sharp\sharp}\cdot a=\lim\limits_{\mu}m_{\mu}a$$
for every $a$ in $\mathcal{A}$ and every $m^{\sharp\sharp}$ in $\mathcal{M}^{\sharp\sharp}$, where
$(m_{\mu})$ is a net in $\mathcal{M}$ with $\|m_{\mu}\|\leqslant\|m^{\sharp\sharp}\|$
and $(m_{\mu})\rightarrow m^{\sharp\sharp}$ in the weak$^{*}$-topology $\sigma(\mathcal{M}^{\sharp\sharp},\mathcal{M}^{\sharp})$.

We define an involution $*$ in $\mathcal M^{\sharp\sharp}$ by
$$(m^{\sharp\sharp})^*(\rho)=\overline{m^{\sharp\sharp}(\rho^*)},~~~~\rho^*(m)=\overline{\rho(m^*)},$$
where
$m^{\sharp\sharp}$ in $\mathcal{M}^{\sharp\sharp}$, $\rho$ in $\mathcal M^{\sharp}$ and $m$ in $\mathcal M$.
Moreover, if $(m_{\mu})$ is a net in $\mathcal{M}$ and $m^{\sharp\sharp}$ is an element in $\mathcal M^{\sharp\sharp}$ such that
$m_\mu\rightarrow m^{\sharp\sharp}$ in $\sigma(\mathcal{M}^{\sharp\sharp},\mathcal{M}^{\sharp})$,
then for every $\rho$ in $\mathcal{M}^{\sharp}$, we have that
$$\rho(m_\mu)=m_\mu(\rho)\rightarrow m^{\sharp\sharp}(\rho).$$
It follows that
$$(m_\mu^*)(\rho)=\rho(m_\mu^*)=\overline{\rho^*(m_\mu)}\rightarrow \overline{m^{\sharp\sharp}(\rho^*)}=(m^{\sharp\sharp})^*(\rho)$$
for every $\rho$ in $\mathcal{M}^{\sharp}$. It means that the involution $*$ in $\mathcal M^{\sharp\sharp}$ is continuous
in $\sigma(\mathcal{M}^{\sharp\sharp}, \mathcal{M}^{\sharp})$.
Thus we can obtain that
$$(a\cdot m^{\sharp\sharp})^*=(\lim\limits_{\mu}am_{\mu})^*=
\lim\limits_{\mu}m_{\mu}^*a^*=(m^{\sharp\sharp})^*\cdot a^*,$$
Similarly, we can show that $(m^{\sharp\sharp}\cdot a)^*=a^*\cdot(m^{\sharp\sharp})^*$.
It implies that $\mathcal M^{\sharp\sharp}$ is a Banach $*$-$\mathcal A$-bimodule.

Let $\mathcal A$ be a Banach $*$-algebra, a \emph{bounded approximate identity} for $\mathcal A$
is a net $(e_{i})_{i\in\Gamma}$ of self-adjoint elements in $\mathcal A$ such that $\lim\limits_{i}\|ae_i-a\|=\lim\limits_{i}\|e_ia-a\|=0$
for every $a$ in $\mathcal A$ and $\mathrm{sup}_{i\in\Gamma}\|e_{i}\|\leq k$ for some $k>0$ .

In \cite{Hoger Ghahramani2}, H. Ghahramani and Z. Pan prove that if $\mathcal A$ is
a unital zero product determined $*$-algebra and a linear mapping $\delta$
from $\mathcal A$ into itself satisfies the condition
$$(\mathbb{D}_{1})~a,b\in\mathcal{A},~ab^*=0\Rightarrow a\delta(b)^*+\delta(a)b^*=0$$
then $\delta(a)=\Delta(a)+\delta(1)a$ for every $a$ in $\mathcal A$, where
$\Delta$ is a $*$-derivation.

For general zero product determined Banach $*$-algebra with a bounded approximate identity, we have the following result.

\begin{theorem}\label{++++5}
Suppose that $\mathcal A$ is a zero product determined Banach $*$-algebra
with a bounded approximate identity, and $\mathcal M$ is an essential Banach $*$-$\mathcal A$-bimodule.
If $\delta$ is a continuous linear mapping from $\mathcal A$
into $\mathcal M$ such that
$$a,b\in\mathcal{A},~ab^*=0\Rightarrow a\delta(b)^*+\delta(a)b^*=0$$
then there exist a $*$-derivation $\Delta$ from $\mathcal A$
into $\mathcal M^{\sharp\sharp}$ and an element $\xi$ in $\mathcal M^{\sharp\sharp}$ such that
$\delta(a)=\Delta(a)+\xi\cdot a$ for every $a$ in $\mathcal A$. Furthermore, $\xi$
can be chosen in $\mathcal M$ in each of the following cases:\\
$\mathrm{(1)}$ $\mathcal A$ is a unital $*$-algebra.\\
$\mathrm{(2)}$ $\mathcal M$ is a dual $*$-$\mathcal A$-bimodule.
\end{theorem}

\begin{proof}
Let $(e_{i})_{i\in\Gamma}$ be a bounded approximate identity of $\mathcal A$.
Since $\delta$ is continuous, the net $(\delta(e_{i}))_{i\in\Gamma}$ is bounded and we can assume that it converges to $\xi$ in $\mathcal M^{\sharp\sharp}$
with the topology $\sigma(\mathcal M^{\sharp\sharp},\mathcal M^{\sharp})$.

Since $\mathcal M$ is an essential Banach $*$-$\mathcal A$-bimodule,
we know that the nets $(e_i m)_{i\in\Gamma}$ and $(me_i)_{i\in\Gamma}$ converge
to $m$ with the norm topology for every $m$ in $\mathcal M$.
Thus we have that
$$\mathrm{Ann}_{\mathcal M}(\mathcal A)=\{m\in\mathcal M:amb=0~\mathrm{for~each}~a,b\in\mathcal A\}=\{0\}.$$
By the hypothesis, we can obtain that
$$a,b,c\in\mathcal{A},~ab^*=b^*c=0\Rightarrow a\delta(b)^*c=0.$$
It follows that
\begin{align}
a,b,c\in\mathcal{A},~ab=bc=0\Rightarrow c^*b^*=b^*a^*=0\Rightarrow c^*\delta(b)^*a^*=0\Rightarrow a\delta(b)c=0.                                                    \label{812-01}
\end{align}
By \eqref{812-01} and \cite[Theorem 4.5]{J. Alaminos}, we know that
$$\delta(ab)=\delta(a)b+a\delta(b)-a\cdot\xi\cdot b$$
for each $a,b$ in $\mathcal A$, and $\xi$
can be chosen in $\mathcal M$ if $\mathcal A$ is a unital $*$-algebra or $\mathcal M$ is a dual $*$-$\mathcal A$-bimodule.

Define a linear mapping $\Delta$ from $\mathcal{A}$
into $\mathcal{M}$ by
$$\Delta(a)=\delta(a)-\xi\cdot a$$
for every $a$ in $\mathcal A$.
It is easy to show that $\Delta$ is a norm-continuous derivation from $\mathcal A$
into $\mathcal M^{\sharp\sharp}$ and we only need
to show that $\Delta(b^*)=\Delta(b)^*$ for every $b$ in $\mathcal A$.

First we claim that $\Delta(e_i)=\delta(e_i)-\xi\cdot e_i$ converges to zero in $\mathcal M^{\sharp\sharp}$
with the topology $\sigma(\mathcal M^{\sharp\sharp},\mathcal M^{\sharp})$.
In fact, since $(e_{i})_{i\in\Gamma}$ is bounded in $\mathcal A$, we assume $(e_{i})_{i\in\Gamma}$
converges to $\zeta$ in $\mathcal A^{\sharp\sharp}$ with the topology $\sigma(\mathcal A^{\sharp\sharp},\mathcal A^{\sharp})$.
For every $m^{\sharp\sharp}$ in $\mathcal M^{\sharp\sharp}$, define
$$m^{\sharp\sharp}\cdot\zeta=\lim\limits_{i}m^{\sharp\sharp}\cdot e_{i}.$$
Thus $m\cdot\zeta=m$ for every $m$ in $\mathcal M$.
By \cite[Proposition A.3.52]{H. Dales}, we know that the mapping $m^{\sharp\sharp}\mapsto m^{\sharp\sharp}\cdot\zeta$
from $\mathcal{M}^{\sharp\sharp}$ into itself is $\sigma(\mathcal M^{\sharp\sharp},\mathcal M^{\sharp})$-continuous,
and by the
$\sigma(\mathcal M^{\sharp\sharp},\mathcal M^{\sharp})$-denseness of $\mathcal M$ in $\mathcal M^{\sharp\sharp}$,
we have that
\begin{align}
m^{\sharp\sharp}\cdot\zeta=m^{\sharp\sharp}                                                                                           \label{813-01}
\end{align}
for every $m^{\sharp\sharp}$ in $\mathcal M^{\sharp\sharp}$. Hence
$\Delta(e_i)=\delta(e_i)-\xi\cdot e_i$ converges to zero in $\mathcal M^{\sharp\sharp}$
with  the topology $\sigma(\mathcal M^{\sharp\sharp},\mathcal M^{\sharp})$.

Next we prove $\Delta(b^*)=\Delta(b)^*$ for every $b$ in $\mathcal A$.
By the definition of $\Delta$, we know that $a\Delta(b)^*+\Delta(a)b^*=0$ for each $a,b$ in $\mathcal A$
with $ab^*=0$.
Define a bilinear mapping from $\mathcal{A}\times\mathcal A$
into $\mathcal{M}^{\sharp\sharp}$ by
$$\phi(a,b)=a\Delta(b^*)^*+\Delta(a)b.$$
Thus $ab=0$ implies $\phi(a,b)=0$.
Since $\mathcal A$ is a zero product determined algebra,
there exists a norm-continuous linear mapping $T$ from $\mathcal A$ into $\mathcal M^{\sharp\sharp}$ such that
\begin{align}
T(ab)=\phi(a,b)=a\Delta(b^*)^*+\Delta(a)b                                                    \label{+++1}
\end{align}
for each $a,b$ in $\mathcal A$.
Let $b=e_i$ be in \eqref{+++1}, we can obtain that
$$T(ae_i)=a\Delta(e_i)^*+\Delta(a)e_i.$$
By the continuity of $T$ and  \eqref{813-01}, it follows that
$T(a)=\Delta(a)$ for every $a$ in $\mathcal A$.
Thus
$$T(ab)=\Delta(ab)=a\Delta(b^*)^*+\Delta(a)b.$$
Since $\Delta$ is a derivation, we have that $a\Delta(b^*)^*=a\Delta(b)$ and $\Delta(b^*)a^*=\Delta(b)^*a^*$.
Let $a=e_i$ and taking $\sigma(\mathcal M^{\sharp\sharp},\mathcal M^{\sharp})$-limits, by \eqref{813-01}, it follows that $\Delta(b^*)=\Delta(b)^*$ for every $b$ in $\mathcal A$.
\end{proof}

Let $G$ be a locally compact group. The group algebra and the measure convolution algebra of $G$,
are denoted by $L^{1}(G)$ and $M(G)$, respectively. The convolution product is denote by $\cdot$
and the involution is denoted by $*$. It is well known that $M(G)$ is a unital Banach $*$-algebra,
and $L^{1}(G)$ is a closed ideal in $M(G)$ with a bounded approximate identity.
By \cite[Lemma 1.1]{J. Alaminos2}, we know that $L^{1}(G)$ is zero product determined.
By \cite[Theorem 3.3.15(ii)]{H. Dales}, it follows that
$M(G)$ with respect to convolution product is the dual of $C_{0}(G)$ as a Banach $M(G)$-bimodule.

By \cite[Corollary 1.2]{Viktor Losert}, we know that every continuous derivation $\Delta$ from $L^{1}(G)$ into $M(G)$
is an inner derivation, that is, there exists $\mu$ in $M(G)$ such that $\Delta(f)=f\cdot\mu-\mu\cdot f$
for every $f$ in $L^{1}(G)$. Thus by Theorem \ref{++++5}, we can prove \cite[Theorem 3.1(ii)]{Hoger Ghahramani1}
as follows.

\begin{corollary}
Let $G$ be a locally compact group. If $\delta$ is a continuous linear mapping from $L^{1}(G)$ into $M(G)$ such that
$$f,g\in L^{1}(G),~f\cdot g^*=0\Rightarrow f\cdot\delta(g)^*+\delta(f)\cdot g^*=0$$
then there are $\mu,\nu$ in $M(G)$ such that
$$\delta(f)=f\cdot\mu-\nu\cdot f$$
for every $f$ in $L^{1}(G)$ and $\mathrm{Re}\mu\in\mathcal Z(M(G))$.
\end{corollary}

\begin{proof}
By Theorem \ref{++++5}, we know that there exist a $*$-derivation $\Delta$ from $L^{1}(G)$ into $M(G)$
and an element $\xi$ in $M(G)$ such that
$\delta(f)=\Delta(f)+\xi\cdot f$ for every $f$ in $L^{1}(G)$.
By \cite[Corollary 1.2]{Viktor Losert}, it follows that there exists $\mu$ in $M(G)$ such that $\Delta(f)=f\cdot\mu-\mu\cdot f$.
Since $\Delta(f^*)=\Delta(f)^*$, we have that
$$f^*\cdot\mu-\mu\cdot f^*=\mu^*\cdot f^*-f^*\cdot\mu^*$$
for every $f$ in $L^{1}(G)$. By \cite[Lemma 1.3(ii)]{J. Alaminos2}, we know $\mathrm{Re}\mu=\frac{1}{2}(\mu+\mu^*)\in\mathcal Z(M(G))$.
Let $\nu=\mu-\xi$, from the definition of $\Delta$, we have that
$\delta(f)=f\cdot\mu-\nu\cdot f$
for every $f$ in $L^{1}(G)$.
\end{proof}

For a general $C^*$-algebra $\mathcal A$, in \cite{B. Fadaee}, B. Fadaee and H. Ghahramani prove that
if $\delta$ is a continuous linear mapping from $\mathcal A$ into its second dual space $\mathcal A^{\sharp\sharp}$
such that the condition $\mathbb{D}_{1}$, then there exist a $*$-derivation $\Delta$ from $\mathcal A$
into $\mathcal A^{\sharp\sharp}$ and an element $\xi$ in $\mathcal A^{\sharp\sharp}$ such that
$\delta(a)=\Delta(a)+\xi a$ for every $a$ in $\mathcal A$.

In \cite{J. Alaminos}, the authors prove that every $C^*$-algebra $\mathcal A$ is zero product determined,
and it is well known that $\mathcal A$ has a bounded approximate identity.
Thus by Theorem \ref{++++5}, we can improve the result in \cite{B. Fadaee} for any
essential Banach $*$-bimodule.

\begin{corollary}\label{new2}
Suppose that $\mathcal A$ is a $C^*$-algebra and $\mathcal M$ is an essential Banach $*$-$\mathcal A$-bimodule.
If $\delta$ is a continuous linear mapping from $\mathcal A$ into $\mathcal M$ such that
$$a,b\in\mathcal{A},~ab^*=0\Rightarrow a\delta(b)^*+\delta(a)b^*=0$$
then there exist a $*$-derivation $\Delta$ from $\mathcal A$
into $\mathcal M^{\sharp\sharp}$ and an element $\xi$ in $\mathcal M^{\sharp\sharp}$ such that
$\delta(a)=\Delta(a)+\xi \cdot a$ for every $a$ in $\mathcal A$.  Furthermore, $\xi$
can be chosen in $\mathcal M$ in each of the following cases:\\
$\mathrm{(1)}$ $\mathcal A$ has an identity.\\
$\mathrm{(2)}$ $\mathcal M$ is a dual $*$-$\mathcal A$-bimodule.
\end{corollary}

For von Neumann algebras, we have the following result.

\begin{theorem}\label{new01}
Suppose that $\mathcal{A}$ is a von Neumann algebra. If $\delta$ is a linear mapping
from $\mathcal{A}$ into itself such that
$$a,b\in\mathcal{A},~ab^*=0\Rightarrow a\delta(b)^*+\delta(a)b^*=0,$$
then $\delta(a)=\Delta(a)+\delta(1)a$ for every $a$ in $\mathcal A$,
where $\Delta$ is a $*$-derivation. In particular, $\delta$ is a $*$-derivation when $\delta(1)=0$.
\end{theorem}

\begin{proof}
Define a linear mapping $\Delta$ from $\mathcal{A}$
into $\mathcal{M}$ by
$$\Delta(a)=\delta(a)-\delta(1)a$$
for every $a$ in $\mathcal A$.
In the following we show that $\Delta$ is a $*$-derivation.
It is clear that $\Delta(1)=0$ and $ab^*=0$ can implies that $a\Delta(b)^*+\Delta(a)b^*=0$.

\textbf{Case 1}: Suppose that $\mathcal A$ is an abelian von Neumann algebra.
First we show that $\Delta$ satisfies that
$$a,b\in\mathcal{A},~ab=0\Rightarrow a\Delta(b)=0.$$

It is well known that
$\mathcal A\cong C(X)$, where $X$ is a compact Hausdorff space and $C(X)$ denotes
the $C^*$-algebra of all continuous complex-valued functions on $X$.
Thus we have that $ab=0$ if and only if $ab^*=0$ for each $a,b$ in $\mathcal{A}$. Indeed,
let $f$ and $g$ be two functions in $C(X)$ corresponding to $a$ and $b$, respectively,
we can obtain that
$$ab^*=0\Leftrightarrow f\cdot\bar{g}=0\Leftrightarrow f\cdot g=0\Leftrightarrow ab=0.$$
Let $a$ and $b$ be in $\mathcal A$ with $ab^*=ab=0$,
we have that $a\Delta(b)^*+\Delta(a)b^*=0$. Multiply $a$ from the left side
of above equation, we can obtain that $a^2\Delta(b)^*=0$.
Let $f$ and $h$ be two functions in $C(X)$ corresponding to $a$ and $\Delta(b)$,
then we have that
$$0=f^2\bar{g}=f^2g=fg.$$
It implies that $a\Delta(b)=0$. By \cite[Theorem 3]{Robert}, we know that $\Delta$
is continuous.
By \cite[Lemma 2.5]{hejun} and $\Delta(1)=0$, we know that $\Delta(a)=\Delta(1)a=0$ for every $a$ in $\mathcal A$.

\textbf{Case 2}: Suppose that $\mathcal A\cong M_n(\mathcal B)$,
where $\mathcal B$ is also a von Neumann algebra and $n\geqslant2$.
By \cite{M. Bresar1, M. Bresar 3} we know that
$\mathcal A$ is a zero product determined algebra. Thus by \cite[Theorem 3.1]{Hoger Ghahramani2}
it follows that $\Delta$ is a $*$-derivation.

\textbf{Case 3}: Suppose that $\mathcal A$ is a general von Neumann algebra.
It is well known that $\mathcal A\cong \sum_{i=1}^{n}\bigoplus\mathcal A_i$ ($n$ is a finite integer or infinite),
where each $\mathcal A_i$ coincides with either Case 1 or Case 2.
Denote the unit element of $\mathcal A_i$ by $1_i$ and the restriction of $\Delta$ in $\mathcal A_i$
by $\Delta_i$. Since $1_i(1-1_i)=0$ and $\Delta(1)=0$, we have that
$$1_i\Delta(1-1_i)^*+\Delta(1_i)(1-1_i)=0.$$
It follows that
\begin{align}
-1_i\Delta(1_i)^*+\Delta(1_i)-\Delta(1_i)1_i=0.                                 \label{new1}
\end{align}
Multiplying $1_i$ from the left side of \eqref{new1} and by $1_i\Delta(1_i)=\Delta(1_i)1_i$, we have that
$1_i\Delta(1_i)^*=0$. It implies that $\Delta(1_i)=0$.
For every $a$ in $\mathcal A$, we write $a=\sum_{i=1}^{n}a_i$
with $a_i$ in $\mathcal A_i$. Since $a_i(1-1_i)=0$, we have that
$\Delta(a_i)(1-1_i)=0$, which means that $\Delta(a_i)\in\mathcal A_i$.
Let $a_i, b_i$ be in $\mathcal A_i$ with
$a_ib_i^*=0$, we have that
$$\Delta(a_i)b_i^*+a_i\Delta(b_i)^*=\Delta_i(a_i)b_i^*+a_i\Delta_i(b_i)^*=0.$$
By Cases 1 and 2, we know that every $\Delta_i$ is a $*$-derivation. Thus $\Delta$ is a $*$-derivation.
\end{proof}

In the following, we characterize
a linear mapping $\delta$ satisfies the
condition $\mathbb{D}_{1}$ from a unital $\ast$-algebra into a unital
$\ast$-$\mathcal A$-bimodule with a right or left separating set $\mathcal{J}\subseteq\mathfrak{J}(\mathcal{A})$.

\begin{lemma} {\rm\cite[Theorem~4.1]{M. Bresar 3}}\label{01}
Suppose that $\mathcal A$ is a unital algebra and $\mathcal X$ is a linear space.
If $\phi$ is a bilinear mapping from $\mathcal{A}\times\mathcal{A}$ into $\mathcal{X}$ such that
$$a,b\in\mathcal{A},~ab=0\Rightarrow\phi(a,b)=0,$$
then we have that
$$\phi(a,x)=\phi(ax,1)~\mathrm{and}~\phi(x,a)=\phi(1,xa)$$
for every $a$ in $\mathcal{A}$ and every $x$ in $\mathfrak{J}(\mathcal{A})$.
\end{lemma}

\begin{theorem}\label{11}
Suppose that $\mathcal A$ is a unital $\ast$-algebra and $\mathcal M$ is a unital
$\ast$-$\mathcal A$-bimodule with a right or left separating set $\mathcal{J}\subseteq\mathfrak{J}(\mathcal{A})$.
If $\delta$ is a linear mapping from $\mathcal{A}$ into $\mathcal{M}$ such that
$$a,b\in\mathcal{A},~ab^*=0\Rightarrow a\delta(b)^*+\delta(a)b^*=0$$
then $\delta(a)=\Delta(a)+\delta(1)a$ for every $a$ in $\mathcal A$,
where $\Delta$ is a $*$-derivation. In particular, $\delta$ is a $*$-derivation when $\delta(1)=0$.
\end{theorem}

\begin{proof}
Since $\mathcal A$ is a unital $\ast$-algebra and $\mathcal M$ is a unital
$\ast$-$\mathcal A$-bimodule, we know that $\mathcal{J}\subseteq\mathfrak{J}(\mathcal{A})$ is a right separating set of $\mathcal{M}$
if and only if $\mathcal{J}^*=\{x^*:x\in\mathcal J\}\subseteq\mathfrak{J}(\mathcal{A})$ is a left separating set of $\mathcal{M}$.
Thus without loss of generality, we can assume that $\mathcal J$ is a left separating set of $\mathcal A$,
otherwise, we replace $\mathcal J$ by $\mathcal{J}^*$.

Define a linear mapping $\Delta$ from $\mathcal{A}$
into $\mathcal{M}$ by
$$\Delta(a)=\delta(a)-\delta(1)a$$
for every $a$ in $\mathcal A$.
In the following we show that $\Delta$ is a $*$-derivation.

It is clear that $\Delta(1)=0$ and $ab^*=0$ can implies that $a\Delta(b)^*+\Delta(a)b^*=0$.
Define a bilinear mapping $\phi$ from $\mathcal{A}\times\mathcal{A}$
into $\mathcal{M}$ by
$$\phi(a,b)=a\Delta(b^*)^*+\Delta(a)b$$
for each $a$ and $b$ in $\mathcal{A}$.
By the assumption we know that $ab=0$ implies $\phi(a,b)=0$.

Let  $a$, $b$ be in $\mathcal{A}$ and $x$ be in $\mathcal{J}$.
By Lemma \ref{01}, we can obtain that
$$\phi(x,1)=\phi(1,x)~\mathrm{and}~\phi(a,x)=\phi(ax,1).$$
Hence we have the following two identities:
\begin{align}
x\Delta(1)^*+\Delta(x)=\Delta(x^*)^*+\Delta(1)x                                           \label{1.1}
\end{align}
and
\begin{align}
a\Delta(x^*)^*+\Delta(a)x=ax\Delta(1)^*+\Delta(ax).                                         \label{1.2}
\end{align}
By \eqref{1.1} and $\Delta(1)=0$, we know that $\Delta(x)^*=\Delta(x^*)$.
Thus by \eqref{1.2}, it implies that
$$\Delta(ax)=a\Delta(x)+\Delta(a)x.$$
Similar to the proof of \cite[Theorem 2.3]{anguangyu}, we can obtain that
$\Delta(ab)=a\Delta(b)+\Delta(a)b$
for each $a$ and $b$ in $\mathcal A$.

It remains to show that $\Delta(a)^*=\Delta(a^*)$
for every $a$ in $\mathcal A$.
Indeed, for every $a$ in $\mathcal{A}$ and every $x$ in $\mathcal{J}$, we have that
$\Delta(ax)^*=\Delta((ax)^*)$. It implies that
$$(\Delta(a)x+a\Delta(x))^*=\Delta(x^*)a^*+x^*\Delta(a^*).$$
Thus we can obtain that
$x^*(\Delta(a)^*-\Delta(a^*))=0$, hence $(\Delta(a)-\Delta(a^*)^*)x=0$.
It follows that $\Delta(a)^*=\Delta(a^*)$ for every $a$ in $\mathcal A$.
\end{proof}

\textbf{Remark 1}. Let $\mathcal A$ be a $*$-algebra, $\mathcal M$ be a $*$-$\mathcal A$-bimodule,
and $\delta$ is a linear mapping from $\mathcal{A}$ into $\mathcal{M}$.
Similar to the condition $\mathbb{D}_{1}$ which we have characterized in Section 2:
$$(\mathbb{D}_{1})~a,b\in\mathcal{A},~ab^*=0\Rightarrow a\delta(b)^*+\delta(a)b^*=0,$$
we can consider the condition $\mathbb{D}_{1}'$
$$(\mathbb{D}_{1}')~a,b\in\mathcal{A},~a^*b=0\Rightarrow a^*\delta(b)+\delta(a)^*b=0.$$
Through the minor modifications, we can obtain the corresponding results.

\textbf{Remark 2}. A linear mapping $\delta$ from $\mathcal{A}$ into
$\mathcal{M}$ is called a \emph{local derivation} if for every $a$ in $\mathcal{A}$,
there exists a derivation $\delta_{a}$ (depending on $a$) from $\mathcal{A}$ into $\mathcal{M}$
such that $\delta(a)=\delta_{a}(a)$.
It is clear that every local derivation satisfies
the following condition:
\begin{align*}
(\mathbb{H})~a,b,c\in\mathcal{A},~ab=bc=0\Rightarrow a\delta(b)c=0.
\end{align*}
In \cite{J. Alaminos}, the authors
prove that every continuous linear mapping from a unital $C^*$-algebra into its unital Banach bimodule
such that the condition $\mathbb{H}$ and $\delta(1)=0$ is a derivation.

Let $\mathcal{A}$ be a $*$-algebra and $\mathcal{M}$ be a $*$-$\mathcal A$-bimodule.
The natural way to translate the condition $\mathbb{H}$
to the context of $*$-derivations is to consider the following condition
\begin{align*}
(\mathbb{H}')~a,b,c\in\mathcal{A},~ab^*=b^*c=0\Rightarrow a\delta(b)^*c=0.
\end{align*}
However, the conditions $\mathbb{H}'$ and $\mathbb{H}$
are equivalent. Indeed, if condition $\mathbb{H}'$ holds, we have that
\begin{align*}
a,b,c\in\mathcal{A},~ab=bc=0\Rightarrow c^*b^*=b^*a^*=0\Rightarrow c^*\delta(b)^*a^*=0\Rightarrow a\delta(b)c=0,
\end{align*}
and if the condition $\mathbb{H}$ holds, we have that
\begin{align*}
a,b,c\in\mathcal{A},~ab^*=b^*c=0\Rightarrow c^*b=ba^*=0\Rightarrow c^*\delta(b)a^*=0\Rightarrow a\delta(b)^*c=0.
\end{align*}
It means that the condition $\mathbb{H}'$
and $\delta(1)=0$
can not implies that $\delta$ is a $*$-derivation.

\section{$*$-Jordan derivations on some algebras}

A (Banach) algebra $\mathcal{A}$ is said to be \emph{zero Jordan product determined}
if every (continuous) bilinear mapping $\phi$ from $\mathcal{A}\times\mathcal{A}$
into any (Banach) linear space $\mathcal{X}$ satisfying
$$\phi(a, b)=0,~\mathrm{ whenever } \  a\circ b = 0$$
can be written as $\phi(a, b)= T(a\circ b)$, for some (continuous) linear mapping $T$
from $\mathcal{A}$ into $\mathcal X$.
In \cite{Guangyu. An}, we show that if $\mathcal A$ is a unital algebra with $\mathcal A=\mathfrak{J}(\mathcal{A})$,
then $\mathcal A$ is a zero Jordan product determined algebra.

\begin{theorem}\label{+++8}
Suppose that $\mathcal A$ is a unital zero Jordan product determined $*$-algebra,
and $\mathcal M$ is a unital $*$-$\mathcal A$-bimodule.
If $\delta$ is a linear mapping from $\mathcal A$
into $\mathcal M$ such that
$$a,b\in\mathcal{A},~a\circ b^*=0\Rightarrow a\circ \delta(b)^*+\delta(a)\circ b^*=0~\mathrm{and}~\delta(1)a=a\delta(1),$$
then $\delta(a)=\Delta(a)+\delta(1)a$ for every $a$ in $\mathcal A$,
where $\Delta$ is a $*$-Jordan derivation. In particular, $\delta$ is a $*$-Jordan derivation when $\delta(1)=0$.
\end{theorem}

\begin{proof}
Define a linear mapping $\Delta$ from $\mathcal{A}$
into $\mathcal{M}$ by
$\Delta(a)=\delta(a)-\delta(1)a$ for every $a$ in $\mathcal A$.
It is sufficient to show that $\Delta$ is a $*$-Jordan derivation.

It is clear that $\Delta(1)=0$, and by $\delta(1)a=a\delta(1)$ we have that
$$a,b\in\mathcal{A},~a\circ b^*=0\Rightarrow a\circ \Delta(b)^*+\Delta(a)\circ b^*=0.$$

Define a bilinear mapping from $\mathcal{A}\times\mathcal A$
into $\mathcal{M}$ by
$$\phi(a,b)=a\circ \Delta(b^*)^*+\Delta(a)\circ b.$$
Thus $a\circ b=0$ implies $\phi(a,b)=0$.
Since $\mathcal A$ is a zero Jordan product determined algebra, we know that
there exists a linear mapping $T$ from $\mathcal A$ into $\mathcal M$ such that
\begin{align}
T(a\circ b)=\phi(a,b)=a\circ \Delta(b^*)^*+\Delta(a)\circ b                                     \label{+++3}
\end{align}
for each $a,b$ in $\mathcal A$.
Let $a=1$ and $b=1$ be in \eqref{+++3}, respectively. By $\Delta(1)=0$, we can obtain that
$$T(a)=\Delta(a)~\mathrm{and}~T(b)=\Delta(b^*)^*.$$
It follows that
$\Delta(a^*)=\Delta(a)^*$ for every $a$ in $\mathcal A$. By \eqref{+++3}, we have that
$$T(a\circ b)=\Delta(a\circ b)=\phi(a,b)=a\circ\Delta(b)+\Delta(a)\circ b.$$
It means that $\Delta$ is a $*$-Jordan derivation.
\end{proof}

In \cite{Guangyu. An}, we  prove that
the matrix algebra $M_n(\mathcal B)(n\geq2)$ is zero Jordan product
determined, where $\mathcal B$ is a unital algebra.
In \cite{Hoger Ghahramani3}, H. Ghahramani show that every Jordan derivation from
$M_n(\mathcal B)(n\geq2)$ into its unital bimodule $\mathcal M$ is a derivation.
Hence we have the following result.

\begin{corollary}\label{06}
Suppose that $\mathcal B$ is a unital $*$-algebra, $M_n(\mathcal B)$ is a matrix algebra with $n\geq2$,
and $\mathcal M$ is a unital $*$-$M_n(\mathcal B)$-bimodule.
If $\delta$ is a linear mapping from $M_n(\mathcal B)$
into $\mathcal M$ such that
$$a,b\in M_n(\mathcal B),~a\circ b^*=0\Rightarrow a\circ \delta(b)^*+\delta(a)\circ b^*=0~\mathrm{and}~\delta(1)a=a\delta(1),$$
then $\delta(a)=\Delta(a)+\delta(1)a$ for every $a$ in $M_n(\mathcal B)$,
where $\Delta$ is a $*$-derivation. In particular, $\delta$ is a $*$-derivation when $\delta(1)=0$.
\end{corollary}

Let $\mathcal H$ be a complex Hilbert space
and $B(\mathcal H)$ be the algebra of all bounded linear operators on $\mathcal H$.
Suppose that $\mathcal A$ is a von Neumann algebra on
$\mathcal{H}$ and $LS(\mathcal A)$ the set of all locally measurable operators
affiliated with the von Neumann algebra $\mathcal A$.

In \cite{M. Muratov}, M. Muratov and V. Chilin prove that $LS(\mathcal A)$
is a unital $*$-algebra and $\mathcal A\subset LS(\mathcal A)$.
By \cite[Proposition 21.20, Exercise 21.18]{T. Lam},
we know that if $\mathcal{A}$ is a von Neumann algebra without direct summand of type $\mathrm{I}_1$,
and $\mathcal B$ is a $*$-algebra with $\mathcal A\subseteq\mathcal B\subseteq LS(\mathcal A)$, then
$\mathcal B\cong \sum_{i=1}^{k}\bigoplus M_{n_i}(\mathcal B_i)$ ($k$ is a finite integer or infinite),
where $\mathcal B_i$ is a unital algebra.
By Theorem \ref{+++8}, we have the following result.

\begin{corollary}\label{110}
Suppose that $\mathcal{A}$ is a von Neumann algebra without direct summand of type $\mathrm{I}_1$,
and $\mathcal B$ is a $*$-algebra with $\mathcal A\subseteq\mathcal B\subseteq LS(\mathcal A)$.
If $\delta$ is a linear mapping
from $\mathcal B$ into $LS(\mathcal A)$ such that
$$a,b\in\mathcal B,~a\circ b^*=0\Rightarrow a\circ \delta(b)^*+\delta(a)\circ b^*=0~\mathrm{and}~\delta(1)a=a\delta(1),$$
then $\delta(a)=\Delta(a)+\delta(1)a$ for every $a$ in $\mathcal B$,
where $\Delta$ is a $*$-Jordan derivation. In particular, $\delta$ is a $*$-Jordan derivation when $\delta(1)=0$.
\end{corollary}

For von Neumann algebras, by Corollary \ref{06} and
similar to the proof of Theorem \ref{new01}, we can easily obtain the following result
and we omit the proof.

\begin{corollary}\label{+++9}
Suppose that $\mathcal{A}$ is a von Neumann algebra. If $\delta$ is a linear mapping
from $\mathcal{A}$ into itself with such that
$$a,b\in\mathcal{A},~a\circ b^*=0\Rightarrow a\circ \delta(b)^*+\delta(a)\circ b^*=0~\mathrm{and}~\delta(1)a=a\delta(1),$$
then $\delta(a)=\Delta(a)+\delta(1)a$ for every $a$ in $\mathcal A$,
where $\Delta$ is a $*$-derivation. In particular, $\delta$ is a $*$-derivation when $\delta(1)=0$.
\end{corollary}

\begin{lemma}{\rm\cite[Theorem~2.1]{Guangyu. An}}\label{02}
Suppose that $\mathcal A$ is a unital algebra and $\mathcal X$ is a linear space.
If $\phi$ is a bilinear mapping from $\mathcal{A}\times\mathcal{A}$
into $\mathcal{X}$ such that
$$a,b\in\mathcal{A},~a\circ b=0\Rightarrow\phi(a,b)=0,$$
then we have that
$$\phi(a,x)=\frac{1}{2}\phi(ax,1)+\frac{1}{2}\phi(xa,1)$$
for every $a$ in $\mathcal{A}$ and every $x$ in $\mathfrak{J}(\mathcal{A})$.
\end{lemma}

Suppose that $\mathcal A$ is a unital algebra and
$\mathcal M$ is a unital $\mathcal{A}$-bimodule
satisfying that
$$\{m\in\mathcal{M}: xmx=0~\mathrm{for~every}~x\in\mathcal{J}\}=\{0\},$$
where $\mathcal J$ is an ideal of $\mathcal A$ linear generated by idempotents in $\mathcal A$.
In \cite[Theorem 4.3]{Hoger Ghahramani}, H. Ghahramani studies the linear mapping $\delta$
from $\mathcal A$ into  $\mathcal M$ satisfies
$$a,b\in\mathcal{A},~a\circ b=0\Rightarrow a\circ\delta(b)+\delta(a)\circ b=0,$$
and show that $\delta$ is a generalized Jordan derivation.
In the following, we suppose that $\mathcal J$ is an ideal of $\mathcal A$
generated algebraically by all
idempotents in $\mathcal{A}$, and have the following result.

\begin{theorem}\label{15}
Suppose that $\mathcal A$ is a unital $\ast$-algebra, $\mathcal M$ is a unital
$\ast$-$\mathcal{A}$-bimodule, and $\mathcal{J}\subseteq\mathfrak{J}(\mathcal{A})$ is an ideal of $\mathcal A$
such that
$$\{m\in\mathcal{M}: xmx=0~\mathrm{for~every}~x\in\mathcal{J}\}=\{0\}.$$
If $\delta$ is a linear mapping from $\mathcal{A}$ into $\mathcal{M}$ such that
$$a,b\in\mathcal{A},~a\circ b^*=0\Rightarrow a\circ\delta(b)^*+\delta(a)\circ b^*=0~\mathrm{and}~\delta(1)a=a\delta(1),$$
then $\delta(a)=\Delta(a)+\delta(1)a$ for every $a$ in $\mathcal A$,
where $\Delta$ is a $*$-Jordan derivation. In particular, $\delta$ is a $*$-Jordan derivation when $\delta(1)=0$.
\end{theorem}

\begin{proof}
Let $\widehat{\mathcal J}$ be an algebra generated algebraically by $\mathcal J$
and $\mathcal J^*$. Since
$\mathcal{J}\subseteq\mathfrak{J}(\mathcal{A})$ is an ideal of $\mathcal A$,
it is easy to show that $\widehat{\mathcal J}\subseteq\mathfrak{J}(\mathcal{A})$ is also
an ideal of $\mathcal A$, and such that
$$\{m\in\mathcal{M}: xmx=0~\mathrm{for~every}~x\in\widehat{\mathcal J}\}=\{0\}.$$
Thus without loss of generality, we can assume that $\mathcal J$ is a self-adjoint ideal of $\mathcal A$,
otherwise, we may replace $\mathcal J$ by $\widehat{\mathcal J}$.

Define a linear mapping $\Delta$ from $\mathcal{A}$
into $\mathcal{M}$ by
$$\Delta(a)=\delta(a)-\delta(1)a$$
for every $a$ in $\mathcal A$.
In the following we show that $\Delta$ is a $*$-derivation.

It is clear that $\Delta(1)=0$, and by $\delta(1)a=a\delta(1)$ we have that
$a\circ b^*=0$ implies that $a\circ\Delta(b)^*+\Delta(a)\circ b^*=0$.

Define a bilinear mapping $\phi$ from $\mathcal{A}\times\mathcal{A}$
into $\mathcal{M}$ by
$$\phi(a,b)=a\circ\Delta(b^*)^*+\Delta(a)\circ b$$
for each $a$ and $b$ in $\mathcal{A}$.
By the assumption we know that $a\circ b=0$ implies $\phi(a,b)=0$.

Let  $a$, $b$ be in $\mathcal{A}$ and $x$ be in $\mathcal{J}$.
By Lemma \ref{02}, we can obtain that
$$\phi(x,1)=\phi(1,x).$$
It follows that
\begin{align}
x\circ\Delta(1)^*+\Delta(x)\circ 1=1\circ\Delta(x^*)^*+\Delta(1)\circ x.                     \label{1.6}
\end{align}
By \eqref{1.6} and $\Delta(1)=0$, we know that $\Delta(x)^*=\Delta(x^*)$.
Again by Lemma \ref{02}, it follows that
\begin{align}
a\circ\Delta(x^*)^*+\Delta(a)\circ x=\frac{1}{2}[\Delta(ax)\circ 1+\Delta(xa)\circ 1].                       \label{1.7}
\end{align}
By \eqref{1.7} and $\Delta(x)^*=\Delta(x^*)$, it is easy to show that
\begin{align}
\Delta(a\circ x)=a\circ\Delta(x)+\Delta(a)\circ x.                                         \label{1.8}
\end{align}

Next, we prove that $\Delta$ is a Jordan derivation.

Define $\{a,m,b\}=amb+bma$ and $\{a,b,m\}=\{m,b,a\}=abm+mba$
for each $a$, $b$ in $\mathcal{A}$ and every $m$ in $\mathcal{M}$.
Let $a$ be in $\mathcal{A}$ and $x, y$ be in $\mathcal{M}$.

By the technique of the proof of \cite[Theorem 4.3]{Hoger Ghahramani} and \eqref{1.8}, we have the following two identities:
\begin{align}
\Delta\{x,a,y\}=\{\Delta(x),a,y\}+\{x,\Delta(a),y\}+\{x,a,\Delta(y)\},                                                                   \label{1.9}
\end{align}
and
\begin{align}
\Delta\{x,a^{2},y\}=\{\Delta(x),a^{2},y\}+\{x,a\circ\Delta(a),y\}+\{x,a^{2},\Delta(y)\}.                                             \label{1.10}
\end{align}
On the other hand, by \eqref{1.9} we have that
\begin{align}
\Delta\{x,a^{2},x\}=\{\Delta(x),a^{2},x\}+\{x,\Delta(a^{2}),x\}+\{x,a^{2},\Delta(x)\}.                                          \label{1.11}
\end{align}
By comparing \eqref{1.10} and \eqref{1.11}, it follows that
$\{x,\Delta(a^{2}),x\}=\{x,a\circ\Delta(a),x\}.$
That is $x(\Delta(a^{2})-a\circ\Delta(a))x=0.$
By the assumption, it implies that
$\Delta(a^{2})-a\circ\Delta(a)=0$ for every $a$ in $\mathcal{A}$.

It remains to show that $\Delta(a)^*=\Delta(a^*)$
for every $a$ in $\mathcal A$.
Indeed, for every $a$ in $\mathcal{A}$ and every $x$ in $\mathcal{J}$, we have that
$\Delta(xax)^*=\Delta((xax)^*)$.
Since $\Delta$ is a Jordan derivation,
it implies that
$$(\Delta(x)ax+x\Delta(a)x+xa\Delta(x))^*=\Delta(x^*)a^*x^*+x^*\Delta(a^*)x^*+x^*a^*\Delta(x^*).$$
Thus we can obtain that
$x^*(\Delta(a)^*-\Delta(a^*))x^*=0$.
Since $\mathcal{J}$ is a self-adjoint ideal of $\mathcal A$,
it follows that $\Delta(a)^*=\Delta(a^*)$.
\end{proof}

Let $\mathcal A$ be a $C^*$-algebra and $\mathcal M$ be a
Banach $\ast$-$\mathcal A$-bimodule. Denote by $\mathcal
A^{\sharp\sharp}$ and $\mathcal M^{\sharp\sharp}$ the second dual
space of $\mathcal A$ and $\mathcal M$, respectively. By \cite[p.26]{H. Dales2}, we can define a product $\diamond$ in
$\mathcal{A}^{\sharp\sharp}$ by
$$a^{\sharp\sharp}\diamond b^{\sharp\sharp}=\lim\limits_{\lambda}\lim\limits_{\mu}\alpha_{\lambda}\beta_{\mu}$$
for each $a^{\sharp\sharp}$, $b^{\sharp\sharp}$ in $\mathcal{A}^{\sharp\sharp}$, where
$(\alpha_{\lambda})$ and $(\beta_{\mu})$ are two nets in $\mathcal{A}$
with $\|\alpha_{\lambda}\|\leqslant\|a^{\sharp\sharp}\|$ and
$\|\beta_{\mu}\|\leqslant\|b^{\sharp\sharp}\|$, such that
$\alpha_{\lambda}\rightarrow a^{\sharp\sharp}$ and $\beta_{\mu}\rightarrow b^{\sharp\sharp}$
in the weak$^{*}$-topology $\sigma(\mathcal{A}^{\sharp\sharp}, \mathcal{A}^{\sharp})$.
Moreover, we can define an involution $*$
in $\mathcal A^{\sharp\sharp}$ by
$$(a^{\sharp\sharp})^*(\rho)=\overline{a^{\sharp\sharp}(\rho^*)},~~~~\rho^*(a)=\overline{\rho(a^*)},$$
where
$a^{\sharp\sharp}$ in $\mathcal{A}^{\sharp\sharp}$, $\rho$ in $A^{\sharp}$ and $a$ in $\mathcal A$.
By \cite[p.726]{R.Kadison J. Ringrose}, we know that $\mathcal{A}^{\sharp\sharp}$ is a von Neumann algebra
under the product $\diamond$ and the involution $*$.

Since $\mathcal M$ is a Banach $\mathcal A$-bimodule,
$\mathcal{M}^{\sharp\sharp}$ turns into a dual Banach $(\mathcal{A}^{\sharp\sharp},\diamond)$-bimodule with the operation defined by
$$a^{\sharp\sharp}\cdot m^{\sharp\sharp}=\lim\limits_{\lambda}\lim\limits_{\mu}a_{\lambda}m_{\mu}~\mathrm{and}~m^{\sharp\sharp}\cdot a^{\sharp\sharp}=\lim\limits_{\mu}\lim\limits_{\lambda}m_{\mu}a_{\lambda}$$
for every $a^{\sharp\sharp}$ in $\mathcal{A}^{\sharp\sharp}$ and every $m^{\sharp\sharp}$ in $\mathcal{M}^{\sharp\sharp}$, where $(a_{\lambda})$ is a net in $\mathcal{A}$ with $\|a_{\lambda}\|\leqslant\|a^{\sharp\sharp}\|$
and $(a_{\lambda})\rightarrow a^{\sharp\sharp}$ in $\sigma(\mathcal{A}^{\sharp\sharp},\mathcal{A}^{\sharp})$, $(m_{\mu})$ is a net in $\mathcal{M}$ with $\|m_{\mu}\|\leqslant\|m^{\sharp\sharp}\|$
and $(m_{\mu})\rightarrow m^{\sharp\sharp}$ in $\sigma(\mathcal{M}^{\sharp\sharp},\mathcal{M}^{\sharp})$.

We remarked, in the discussion preceding Theorem \ref{++++5}, that $\mathcal{M}^{\sharp\sharp}$  has an involution $*$ and it is continuous
in $\sigma(\mathcal{M}^{\sharp\sharp}, \mathcal{M}^{\sharp})$.
By \cite[p.553]{J. Alaminos}, we know that every continuous bilinear map $\varphi$ from $\mathcal{A}\times\mathcal{M}$ into $\mathcal{M}$ is Arens regular, which means that
$$\lim\limits_{\lambda}\lim\limits_{\mu}\varphi(a_{\lambda},m_{\mu})=\lim\limits_{\mu}\lim\limits_{\lambda}\varphi(a_{\lambda},m_{\mu})$$
for every $\sigma(\mathcal{A}^{\sharp\sharp}, \mathcal{A}^{\sharp})$-convergent net $(a_{\lambda})$ in $\mathcal{A}$ and every $\sigma(\mathcal{M}^{\sharp\sharp}, \mathcal{M}^{\sharp})$-convergent net $(m_{\mu})$ in $\mathcal{M}$.
Thus we can obtain that
$$(a^{\sharp\sharp}\cdot m^{\sharp\sharp})^*=(\lim\limits_{\lambda}\lim\limits_{\mu}a_{\lambda}m_{\mu})^*=
\lim\limits_{\lambda}\lim\limits_{\mu}m_{\mu}^*a_{\lambda}^*=\lim\limits_{\mu}\lim\limits_{\lambda}m_{\mu}^*a_{\lambda}^*=(m^{\sharp\sharp})^*\cdot(a^{\sharp\sharp})^*,$$
where $(a_{\lambda})$ is a net in $\mathcal{A}$ with
$(a_{\lambda})\rightarrow a^{\sharp\sharp}$ in $\sigma(\mathcal{A}^{\sharp\sharp},\mathcal{A}^{\sharp})$ and $(m_{\mu})$ is a net in $\mathcal{M}$ with
$(m_{\mu})\rightarrow m^{\sharp\sharp}$ in $\sigma(\mathcal{M}^{\sharp\sharp},\mathcal{M}^{\sharp})$.
Similarly, we can show that $(m^{\sharp\sharp}\cdot a^{\sharp\sharp})^*=(a^{\sharp\sharp})^*\cdot(m^{\sharp\sharp})^*$.
It implies that $\mathcal M^{\sharp\sharp}$ is a Banach $\ast$-$\mathcal A^{\sharp\sharp}$-bimodule.

A projection $p$ in $\mathcal A^{\sharp\sharp}$ is called \emph{open} if
there exists an increasing net $(a_{\alpha})$ of positive elements in $\mathcal A$ such that $p=\lim\limits_{\alpha}a_{\alpha}$ in the
weak$^*$-topology of $\mathcal A^{\sharp\sharp}$. If $p$ is open, we say the projection $1-p$ is \emph{closed}.

For a unital $C^*$-algebra, we have the following result.

\begin{theorem}\label{04}
Suppose that $\mathcal A$ is a unital $C^*$-algebra and $\mathcal M$ is
a unital Banach $*$-$\mathcal A$-bimodule.
If $\delta$ is a continuous linear mapping from $\mathcal A$
into $\mathcal M$ such that $\delta(1)a=a\delta(1)$ for every $a$ in $\mathcal A$,
then the following three statements are equivalent:\\
$(1)$ $a,b\in\mathcal{A},~a\circ b^*=0\Rightarrow a\circ\delta(b)^*+\delta(a)\circ b^*=0;$\\
$(2)$ $a,b\in\mathcal{A},~ab^*=b^*a=0\Rightarrow a\circ\delta(b)^*+\delta(a)\circ b^*=0;$\\
$(3)$ $\delta(a)=\Delta(a)+\delta(1)a$ for every a in $\mathcal A$,
where $\Delta$ is a $*$-derivation from $\mathcal A$ into $\mathcal M$.
\end{theorem}

\begin{proof}
It is clear that $(1)\Rightarrow(2)$ and $(3)\Rightarrow(1)$. It is sufficient the
prove that $(2)\Rightarrow(3)$.

Define a linear mapping $\Delta$ from $\mathcal A$ into $\mathcal M$
by $\Delta(a)=\delta(a)-\delta(1)a$ for every $a$ in $\mathcal A$.
It is sufficient to show that $\Delta$ is a $*$-derivation.
First we prove that $\Delta(a^*)=\Delta(a)^*$ for every $a$ in $\mathcal A$.

By assumption, we can easily to show that
$$a,b\in\mathcal{A},~ab^*=b^*a=0\Rightarrow a\circ\Delta(b)^*+\Delta(a)\circ b^*=0~\mathrm{and}~\Delta(1)=0,$$
In the following, we verify $\Delta(b)=\Delta(b)^*$ for every self-adjoint element $b$ in $\mathcal A$.

Since $\Delta$ is a norm continuous linear mapping form $\mathcal A$
into $\mathcal M$, we know that $\Delta^{\sharp\sharp}:(\mathcal{A}^{\sharp\sharp}, \diamond)\rightarrow\mathcal{M}^{\sharp\sharp}$ is the
weak$^{*}$-continuous extension
of $\Delta$ to the double duals of $\mathcal{A}$ and $\mathcal{M}$.

Let $b$ be a non-zero self-adjoint element in $\mathcal A$,
$\sigma(b)\subseteq[-\|b\|,\|b\|]$
be the spectrum of $b$ and $r(b)\in\mathcal A^{\sharp\sharp}$ be the range projection of $b$.

Denote by $\mathcal A_b$ the $C^*$-subalgebra of $\mathcal A$
generated by $b$, and by $C(\sigma(b))$
the $C^*$-algebra of all continuous complex-valued functions on $\sigma(b)$.
By Gelfand theory we know that there is an isometric $*$ isomorphism between
$\mathcal A_b$ and $C(\sigma(b))$.

For every $n$ in $\mathbb{N}$, let $p_n$ be the projection in $\mathcal A_b^{\sharp\sharp}\subseteq\mathcal A^{\sharp\sharp}$
corresponding to the characteristic function $\chi_{([-\|b\|,-\frac{1}{n}]\cup[\frac{1}{n},\|b\|])\cap\sigma(b)}$
in $C(\sigma(b))$, and let $b_n$ be in $\mathcal A_b$ such that
$$b_np_n=p_nb_n=b_n=b_n^*~\mathrm{and}~\|b_n-b\|<\frac{1}{n}.$$
By \cite[Section 1.8]{sakai}, we know that $(p_n)$ converges to $r(b)$ in the strong$^*$-topology
of $\mathcal A^{\sharp\sharp}$, and hence in the weak$^*$-topology.

It is well known that $p_n$ is a closed projection in $\mathcal A_b^{\sharp\sharp}\subseteq\mathcal A^{\sharp\sharp}$ and
$1-p_n$ is an open projection in $\mathcal A_b^{\sharp\sharp}$. Thus there exists an increasing net $(z_\lambda)$ of
positive elements in $((1-p_n)\mathcal A^{\sharp\sharp}(1-p_n))\cap\mathcal A$ such that
$$0\leq z_\lambda\leq 1-p_n$$
and $(z_\lambda)$
converges to $1-p_n$ in the weak$^*$-topology of $\mathcal A^{\sharp\sharp}$.
Since
$$0\leq ((1-p_n)-z_\lambda)^2\leq (1-p_n)-z_\lambda\leq(1-p_n),$$
we have that
$(z_\lambda)$ also converges to $1-p_n$ in the strong$^*$-topology of $\mathcal A^{\sharp\sharp}$.

By $b_n=b_n^*$ and $z_{\lambda}b_n=b_nz_{\lambda}=0$, it follows that
\begin{align}
z_{\lambda}\circ\Delta^{\sharp\sharp}(b_n)^*+\Delta^{\sharp\sharp}(z_{\lambda})\circ b_n=0.                                                                                     \label{1.13}
\end{align}
Taking weak$^*$-limits in \eqref{1.13} and since $\Delta^{\sharp\sharp}$ is weak$^*$-continuous, we have that
\begin{align}
(1-p_n)\circ\Delta^{\sharp\sharp}(b_n)^*+\Delta^{\sharp\sharp}((1-p_n))\circ b_n=0.                                                                             \label{1.14}
\end{align}
Since $(p_n)$
converges to $r(b)$ in the weak$^*$-topology of $\mathcal A^{\sharp\sharp}$ and $(b_n)$ converges to $b$
in the norm-topology of $\mathcal A$,
by \eqref{1.14}, we have that
\begin{align}
(1-r(b))\circ\Delta^{\sharp\sharp}(b)^*+\Delta^{\sharp\sharp}(1-r(b))\circ b=0.                                     \label{1.15}
\end{align}
Since the range projection of every power $b^m$ with $m\in\mathbb{N}$ coincides with
the $r(b)$, and by \eqref{1.15}, it follows that
\begin{align*}
(1-r(b))\circ\Delta^{\sharp\sharp}(b^m)^*+\Delta^{\sharp\sharp}(1-r(b))\circ b^m=0
\end{align*}
for every $m\in\mathbb{N}$, and by the linearity and norm continuity of the product we have that
\begin{align*}
(1-r(b))\circ\Delta^{\sharp\sharp}(z)^*+\Delta^{\sharp\sharp}(1-r(b))\circ z=0
\end{align*}
for every $z=z^*$ in $\mathcal A_b$. A standard argument involving weak$^*$-continuity of $\Delta^{\sharp\sharp}$ gives
\begin{align}
(1-r(b))\circ\Delta^{\sharp\sharp}(r(b))^*+\Delta^{\sharp\sharp}(1-r(b))\circ r(b)=0.                                 \label{+1}
\end{align}
By \eqref{+1}, we can obtain that
$$
(\Delta^{\sharp\sharp}(r(b))^*+\Delta^{\sharp\sharp}(r(b))-\Delta^{\sharp\sharp}(1))\circ r(b)=2\Delta^{\sharp\sharp}(r(b))^*.
$$
By $\Delta(1)=0$, we have that $\Delta^{\sharp\sharp}(1)=0$. It implies that
\begin{align}
\Delta^{\sharp\sharp}(r(b))^*=\Delta^{\sharp\sharp}(r(b)).                                                                                                                       \label{1.16}
\end{align}

It is clear that every characteristic function
\begin{align}
p=\chi_{([-\|b\|,-\alpha]\cup[\alpha,\|b\|])\cap\sigma(b)}                                              \label{2.18}
\end{align}
in $C_0(\sigma(b))^{\sharp\sharp}$ with $0<\alpha<\|b\|$, is
the range projection of a function in $C(\sigma(b))$.
Moreover, every projection of the form
\begin{align}
q=\chi_{([-\beta,-\alpha]\cup[\alpha,\beta])\cap\sigma(b)}                                              \label{2.19}
\end{align}
in $C_0(\sigma(b))^{\sharp\sharp}$ with $0<\alpha<\beta<\|b\|$ can be written as the
difference of two projections of the type in \eqref{2.18}.

Since $\mathcal A_b$ and $C(\sigma(b))$ are
isometric $*$ isomorphism, and by $\Delta^{\sharp\sharp}(r(b))^*=\Delta^{\sharp\sharp}(r(b))$
for range projection of $b$ in $\mathcal A^{\sharp\sharp}$, we have that $\Delta^{\sharp\sharp}(p)^*=\Delta^{\sharp\sharp}(p)$
for every projection $p$ of the type in \eqref{2.18}. It follows that $\Delta^{\sharp\sharp}(q)^*=\Delta^{\sharp\sharp}(q)$
for every projection $q$ of the type in \eqref{2.19}.

It is well known that $b$ can be approximated in norm by finite linear combinations of mutually orthogonal
projections $q_j$ of the type in \eqref{2.19}, and $\Delta$ is continuous, we have that $\Delta(b)^*=\Delta(b)$.
Thus for every $a$ in $\mathcal A$, we can obtain that $\Delta(a)^*=\Delta(a)$.

By the assumption, it follows that
$$a,b\in\mathcal{A},~ab=ba=0\Rightarrow a\circ\Delta(b)+\Delta(a)\circ b=0.$$
By \cite[Theorem 4.1]{J. Alaminos1}, we know that $\Delta$ is a $*$-derivation.
\end{proof}

In the following we consider general $C^*$-algebras $\mathcal A$.
Let $(e_{i})_{i\in\Gamma}$ be a bounded approximate identity of $\mathcal A$,
$\mathcal M$ be an essential Banach $*$-$\mathcal A$-bimodule, and $\delta$ be a continuous linear mapping from $\mathcal A$
into $\mathcal M$, then $(\delta(e_{i}))_{i\in\Gamma}$ is bounded
and we can assume that it converges to $\xi$ in $\mathcal M^{\sharp\sharp}$
with the topology $\sigma(\mathcal M^{\sharp\sharp},\mathcal M^{\sharp})$.
It follows the next result.

\begin{theorem}
Suppose that $\mathcal A$ is a $C^*$-algebra (not necessary unital) and $\mathcal M$ is
an essential Banach $*$-$\mathcal A$-bimodule.
If $\delta$ is a continuous linear mapping from $\mathcal A$
into $\mathcal M$ such that $\xi\cdot a=a\cdot\xi$ for every $a$ in $\mathcal A$,
then the following three statements are equivalent:\\
$(1)$ $a,b\in\mathcal{A},~a\circ b^*=0\Rightarrow a\circ\delta(b)^*+\delta(a)\circ b^*=0;$\\
$(2)$ $a,b\in\mathcal{A},~ab^*=b^*a=0\Rightarrow a\circ\delta(b)^*+\delta(a)\circ b^*=0;$\\
$(3)$ $\delta(a)=\Delta(a)+\xi\cdot a$ for every a in $\mathcal A$,
where $\Delta$ is a $*$-derivation from $\mathcal A$ into $\mathcal M^{\sharp\sharp}$.
\end{theorem}

\begin{proof}
It is clear that $(1)\Rightarrow(2)$ and $(3)\Rightarrow(1)$. It is only need to
prove that $(2)\Rightarrow(3)$.

Define a linear mapping $\Delta$ from $\mathcal{A}$
into $\mathcal M^{\sharp\sharp}$ by
$$\Delta(a)=\delta(a)-\xi\cdot a$$
for every $a$ in $\mathcal A$.
It is sufficient to show that $\Delta$ is a $*$-derivation.

By the definition of $\Delta$ and $\xi\cdot a=a\cdot\xi$ for every $a$ in $\mathcal A$, we can easily to show that
$$a,b\in\mathcal{A},~ab^*=b^*a=0\Rightarrow a\circ\Delta(b)^*+\Delta(a)\circ b^*=0.$$
By \cite[Proposition 2.9.16]{H. Dales}, we know that $(e_{i})_{i\in\Gamma}$
converges to the identity $1$ in $\mathcal A^{\sharp\sharp}$
with the topology $\sigma(\mathcal A^{\sharp\sharp},\mathcal A^{\sharp})$.
By the proof of Theorem \ref{++++5}, we know that $\Delta(e_i)=\delta(e_i)-e_i\cdot\xi$ converges to zero in $\mathcal M^{\sharp\sharp}$
with the topology $\sigma(\mathcal M^{\sharp\sharp},\mathcal M^{\sharp})$,
and we can obtain that
\begin{align*}
m^{\sharp\sharp}\cdot1=m^{\sharp\sharp}
\end{align*}
for every $m^{\sharp\sharp}$ in $\mathcal{M}^{\sharp\sharp}$. Since
$\mathcal{M}^{\sharp\sharp}$ is a Banach $*$-$\mathcal A^{\sharp\sharp}$-bimodule, we have that
\begin{align*}
1\cdot m^{\sharp\sharp}=m^{\sharp\sharp}
\end{align*}
for every $m^{\sharp\sharp}$ in $\mathcal{M}^{\sharp\sharp}$.
Since $\Delta$ is a norm-continuous linear mapping form $\mathcal A$
into $\mathcal M^{\sharp\sharp}$, $\Delta^{\sharp\sharp}:(\mathcal{A}^{\sharp\sharp}, \diamond)\rightarrow\mathcal{M}^{\sharp\sharp\sharp\sharp}$ is the
weak$^{*}$-continuous extension
of $\Delta$ to the double duals of $\mathcal{A}$ and $\mathcal{M}^{\sharp\sharp}$ such that
$\Delta^{\sharp\sharp}(1)=0$.

By \cite[Proposition A.3.52]{H. Dales}, we know that the mapping $m^{\sharp\sharp\sharp\sharp}\mapsto m^{\sharp\sharp\sharp\sharp}\cdot 1$
from $\mathcal{M}^{\sharp\sharp\sharp\sharp}$ into itself is $\sigma(\mathcal M^{\sharp\sharp\sharp\sharp},\mathcal M^{\sharp\sharp\sharp})$-continuous,
and by the
$\sigma(\mathcal M^{\sharp\sharp\sharp\sharp},\mathcal M^{\sharp\sharp\sharp})$-denseness of $\mathcal M^{\sharp\sharp}$
in $\mathcal M^{\sharp\sharp\sharp\sharp}$,
we have that
\begin{align*}
m^{\sharp\sharp\sharp\sharp}\cdot 1=m^{\sharp\sharp\sharp\sharp}                  
\end{align*}
for every $m^{\sharp\sharp\sharp\sharp}$ in $\mathcal M^{\sharp\sharp\sharp\sharp}$.
Since $\mathcal{M}^{\sharp\sharp\sharp\sharp}$ is a Banach $*$-$\mathcal A^{\sharp\sharp}$-bimodule, we have that
\begin{align*}
1\cdot m^{\sharp\sharp\sharp\sharp}=m^{\sharp\sharp\sharp\sharp}
\end{align*}
for every $m^{\sharp\sharp\sharp\sharp}$ in $\mathcal{M}^{\sharp\sharp\sharp\sharp}$.

Finally, we use the same proof of Theorem \ref{04} and show that $\Delta$ is a $*$-derivation from $\mathcal{A}$
into $\mathcal M^{\sharp\sharp}$.
\end{proof}

\textbf{Remark 3}. In \cite{A. Essaleh}, A. Essaleh and A. Peralta introduce the concept of a triple derivation on $C^*$-algebras.
Suppose that $\mathcal A$ is a $C^*$-algebra. Let $a$, $b$ and $c$ be in $\mathcal A$,
define the \emph{ternary product} by
$\{a,b,c\}=\frac{1}{2}(ab^*c+cb^*a)$.
A linear mapping $\delta$ from $\mathcal A$ into itself is called a \emph{triple derivation} if
$$\delta\{a,b,c\}=\{\delta(a),b,c\}+\{a,\delta(b),c\}+\{a,b,\delta(c)\}$$
for each $a$, $b$ and $c$ in $\mathcal A$. Let $z$ be an element in $\mathcal A$.
$\delta$ is called \emph{triple derivation at $z$} if
$$a,b,c\in\mathcal{A},~\{a,b,c\}=z\Rightarrow \delta(z)=\{\delta(a),b,c\}+\{a,\delta(b),c\}+\{a,b,\delta(c)\}.$$
In \cite{A. Essaleh}, A. Essaleh and A. Peralta
prove that every continuous linear mapping $\delta$ which is triple derivations at zero
from a unital $C^*$-algebra into itself with $\delta(1)=0$ is a $*$-derivation.

On the other hand, it is apparent to show that if $\delta$ is triple derivation at zero, then $\delta$ satisfies that
$$a,b\in\mathcal{A},~ab^*=b^*a=0\Rightarrow a\circ\delta(b)^*+\delta(a)\circ b^*=0.$$
Thus Theorem \ref{04} generalizes \cite[Corollary 2.10]{A. Essaleh}.

\textbf{Remark 4}. In \cite{M. Bresar J. Vukman},
M. Bre\v{s}ar and J. Vukman introduce the left derivations and Jordan left derivations.
A linear mapping $\delta$ from an algebra $\mathcal{A}$ into its bimodule $\mathcal{M}$ is
called a \emph{left derivation} if $\delta(ab)=a\delta(b)+b\delta(a)$ for each $a,b$ in $\mathcal{A}$; and
$\delta$ is called a \emph{Jordan left derivation} if $\delta(a\circ b)=2a\delta(b)+2b\delta(a)$ for each $a,b$ in $\mathcal{A}$.

Let $\mathcal A$ be a $*$-algebra and $\mathcal M$ be a $*$-$\mathcal A$-bimodule.
A left derivation (Jordan left derivation) $\delta$ from $\mathcal{A}$ into $\mathcal{M}$ is called a
\emph{$*$-left derivation} (\emph{$*$-Jordan left derivation})  if $\delta(a^*)=\delta(a)^*$ for every $a$ in $\mathcal A$.

We also can investigate
the following conditions on a linear mapping $\delta$ from $\mathcal{A}$ into $\mathcal{M}$:
\begin{align*}
&(\mathbb{J}_{1})~a,b\in\mathcal{A},~ab^*=0\Rightarrow a\delta(b)^*+b^*\delta(a)=0;\\
&(\mathbb{J}_{2})~a,b\in\mathcal{A},~a\circ b^*=0\Rightarrow a\delta(b)^*+b^*\delta(a)=0;\\
&(\mathbb{J}_{3})~a,b\in\mathcal{A},~ab^*=b^*a=0\Rightarrow a\delta(b)^*+b^*\delta(a)=0.
\end{align*}

{\bf Acknowledgement.}
This research was partly supported by the National Natural Science
Foundation of China (Grant Nos. 11801342, 11801005, 11871021);
Natural Science Foundation of Shaanxi Province (Grant No. 2020JQ-693);
Scientific research plan projects of Shannxi Education Department (Grant No. 19JK0130).


\end{document}